\documentclass[12pt]{article}
\usepackage[inner=2.5 cm,outer=2.0 cm,top=2.0 cm,bottom=2.0 cm]{geometry}

\usepackage{setspace}
\usepackage{geometry}
\usepackage{amsmath,amssymb,amsfonts,amsthm}
\usepackage[english]{babel}
\usepackage[latin1]{inputenc}
\usepackage{times}
\usepackage[T1]{fontenc}
\usepackage{graphicx}
\usepackage{amsfonts}
\usepackage[all]{xy}

\numberwithin{equation}{section}

\newcounter{item}[section]
\newcounter{kirshr}
\newcounter{kirsha}
\newcounter{kirshb}

\newtheorem{theorem}{Theorem}[section]
\newtheorem{proposition}[theorem]{Proposition}

\newtheorem{lemma}[theorem]{Lemma}
\newtheorem{corollary}[theorem]{Corollary}

\theoremstyle{definition}

\newtheorem{definition}[theorem]{Definition}

\def\(R)RA{{\bf (R)RA}}

\title{A Characterization of Modules with Cyclic Socle}
\author{Ali Assem\\\small{Department of Mathematics}\\\small{aassem@sci.cu.edu.eg}}
\begin{document}
\maketitle
\begin{abstract}In 2009, J. Wood \cite{r7} proved that Frobenius bimodules have the extension property for symmetrized weight compositions. Later, in \cite{r4}, it was proved that having a cyclic socle is sufficient for satisfying the property, while the necessity remained an open question.
\newline Here, landing in  Midway, the necessity is proved, a module alphabet $_RA$ has the extension property for symmetrized weight compositions built on $\mathrm{Aut}_R(A)$ is necessarily having a cyclic socle.
 \end{abstract}
\textbf{Note:} All rings are finite with unity, and all modules are finite too. This may be re-emphasized in some statements. The convention for functions is that inputs are to the left.

\setcounter{section}{0}
\section{Introduction}

A (left) linear code of length $n$ over a module alphabet $_RA$ is a (left) submodule $C\subset A^n$. $A$ has the \emph{extension property} (EP) for the weight $w$ if for any $n$ and any two codes $C_1,C_2\subset A^n$, any isomorphism $f: C_1\rightarrow C_2$ preserving $w$ extends to a monomial transformation of $A^n$. In 1962, MacWilliams \cite{Mac} proved the Hamming weight EP  for linear codes over finite fields; in 1996, H. Ward and J. Wood \cite{r1996} re-proved this using the linear independence of group characters.
This kind of proofs -- using characters -- led to further generalities. In 1997, J. Wood \cite{r5} proved that Frobenius rings have the EP for \emph{symmetrized weight compositions} (swc), and in his 1999-paper \cite{duality}, proved that Frobenius rings have the property for Hamming weight. Besides, for the last case, a partial converse was proved: commutative rings satisfying the EP for Hamming weight are necessarily Frobenius.

In 2004, Greferath et al.\cite{2004} showed that Frobenius bimodules do have the EP for Hamming weight. In \cite{r2}, Dinh and L\'{o}pez-Permouth suggested a strategy for proving the full converse. The strategy has three parts. (1) If a finite ring is not Frobenius, its socle contains a matrix module of a particular type. (2) Provide a counter-example to the EP in the context of linear codes over this special module. (3) Show that this counter example over the matrix module pulls back to give a counter example over the original ring.  Finally, in 2008, J. Wood \cite{r6} provided the main technical result for carrying out the strategy, and thereby proving that rings having the EP for Hamming weight are necessarily Frobenius. The proof was easily adapted in \cite{r7} (2009) to prove that a module alphabet $_RA$ has the EP for Hamming weight if and only if $A$ is pseudo-injective with cyclic socle.

On the other lane, in \cite{r7}, J. Wood proved that Frobenius bimodules have the EP for swc, and in \cite{r4} it was shown that having a cyclic socle is sufficient (Theorem \ref{Noha}), while the necessity remained an open question. Here, the necessity is proved,  making use of a new notion, namely, the \emph{annihilator weight}, defined in section 4 below.

\section{Background in Ring Theory}
Let $R$ be a finite ring with unity, denote by $\mathrm{rad}R$ its Jacobson radical, by the Wedderburn-Artin theorem (and Wedderburn's little theorem) the ring $R/\mathrm{rad}R$ is semi-simple, and (as rings)
\begin{equation}\label{iso1}R/\mathrm{rad}R\cong \overset{k}{\underset{i=1}\bigoplus}\mathbb{M}_{\mu_i}(\mathbb{F}_{q_i}),\end{equation}
where each $q_i$ is a prime power, $\mathbb{F}_{q_i}$ denotes a finite field of order $q_i$, and $\mathbb{M}_{\mu_i}(\mathbb{F}_{q_i})$ denotes the ring of $\mu_i\times\mu_i$ matrices over $\mathbb{F}_{q_i}$.

It follows that, as left $R$-modules, \begin{equation}\label{iso2}_R(R/\mathrm{rad}R)\cong \overset{k}{\underset{i=1}\bigoplus}\mu_iT_i,\end{equation} where $_RT_i$ is the pullback to $R$ of the matrix module $_{\mathbb{M}_{\mu_i}(\mathbb{F}_{q_i})}\mathbb{M}_{\mu_i\times 1}(\mathbb{F}_{q_i})$ via the isomorphism in equation (\ref{iso1}). It is known that these $T_i$'s  form the  complete list, up to isomorphism, of all simple left $R$-modules, hence the socle of any $R$-module $A$ can be expressed as $$\mathrm{soc}(A)\cong\overset{k}{\underset{i=1}\bigoplus}s_iT_i,$$ where $s_i$ is the number of copies of $T_i$ inside $A$.

\setlength{\parindent}{0cm}The next two propositions can be found in \cite{r7}, page 17.

\begin{proposition}\label{prop}$\mathrm{soc}(A)$ is cyclic if and only if $s_i\leq\mu_i$ for $i=1,\ldots,k;\;\mu_i$ defined as above.\end{proposition}
\begin{proposition}$\mathrm{soc}(A)$ is cyclic if and only if $A$ can be embedded into $_R\widehat{R}$, the character group of $R$ equipped with the standard module structure.\end{proposition}
\setlength{\parindent}{0.4cm}
\vspace{0.5cm}

The next theorem (Theorem 4.1, \cite{r6}), by J. Wood, was the key to carry out the strategy of Dinh and L\'{o}pez-Permouth mentioned in the introduction, actually, it displays a thoughtfully constructed piece-of-art example for the failure of the Hamming weight EP.
\begin{theorem}\label{wood}Let $R=\mathbb{M}_m(\mathbb{F}_q)$ and $A=\mathbb{M}_{m\times k}(\mathbb{F}_q)$. If $k>m$, there exist linear codes $C_+,C_-\subset A^N, N=\overset{k-1}{\underset{i=1}{\prod}}(1+q^i)$, and an $R$-linear isomorphism $f:C_+\rightarrow C_-$ that preserves Hamming weight, yet there is no monomial transformation extending $f$.\end{theorem}If $\mathrm{soc}(A)$ is not cyclic, then the previous theorem, applied to a certain submodule of $\mathrm{soc}(A)$, gives counter-examples that pull back to give counter-examples for the original module, as the proof of the following theorem shows (a detailed proof is found in \cite{r7}, Theorem 6.4).

\begin{theorem}(Th. 5.2, \cite{r6}). Let $R$ be a finite ring, and let $A$ be a finite left $R$-module. If there exists an index $i$ and a multiplicity $k>\mu_i$ so that $kT_i\subset \mathrm{soc}(A)\subset A$, then the extension property for Hamming weight fails for linear codes over the module $A$.\end{theorem}

\section{Symmetrized Weight Compositions}

\setlength{\parindent}{0.4cm}

\begin{definition}(Symmetrized Weight Compositions) Let $G$ be a subgroup of the automorphism group $\mathrm{Aut}_R(A)$ of a finite $R$-module $A$. Define an equivalence relation $\sim$ on $A$:  $a\sim b$ if $a=b\tau$ for some $\tau \in G$. Let $A/G$ denote the orbit space of this relation. The \emph{symmetrized weight composition} (swc) built on $G$ is a function

swc : $A^n \times A/G\rightarrow \mathbb{Q}$ defined by, $$\mathrm{swc}(x, a) = |\{i:x_i\sim a\}|,$$where $x=(x_1,\ldots,x_n)\in A^n$ and $a\in A/G$. Thus, swc counts the number of components in each orbit.\end{definition}

\begin{definition}(Monomial Transformation) Let $G$ be a subgroup of $\mathrm{Aut}_R(A)$, a map $T$ is called a $G$-\emph{monomial transformation} of $ A^n$ if there are some $\sigma\in S_n$ and $\tau_i\in G$ for $i=1,\ldots,n$, such that $$(x_1,\ldots,x_n)T=(x_{\sigma(1)}\tau_1,\ldots,x_{\sigma(n)}\tau_n),$$where  $(x_1,\ldots,x_n)\in A^n$.\end{definition}

\begin{definition}(Extension Property)  The alphabet $A$ has the \emph{extension
property} (EP) with respect to swc if for every $n$, and any two linear codes $ C_1, C_2\subset A^n $,  any
$R$-linear isomorphism $ f:C_1 \rightarrow C_2$   preserving swc is extends to a
$G$-monomial transformation of $ A^n$.\end{definition}

In \cite{r5}, J.A.Wood proved that Frobenius rings do have the extension property with respect to swc. Later, in \cite{r4}, it was shown that, more generally, a left $R$-module $A$ has the extension property with respect to swc if it can be embedded in the character group $\widehat{R}$ (given the standard module structure).

\begin{theorem}\label{Noha}(Th.4.1.3, \cite{r3}) Let $A$ be a finite left $R$-module. If A can be embedded into  $\widehat{R}$ (or equivalently, $\mathrm{soc}(A)$ is cyclic), then $A$ has the extension
property with respect to the swc built on any subgroup $G$ of $\mathrm{Aut}_R(A)$. In particular, this theorem applies to Frobenius bimodules.\end{theorem}

\section{Annihilator Weight}

We now define a \emph{new }notion (the Midway!) on which we'll depend in the rest of this paper.
\begin{definition}(Annihilator Weight) On $_RA$, define an equivalence relation $\approx$ by $a\approx b$ if $\mathrm{Ann}_a=\mathrm{Ann}_b$, where $a$ and $b$ are any two elements in $A$ and $\mathrm{Ann}_a=\{r\in R{}|{} ra=0\}$ is the annihilator of $a$. Clearly, $\mathrm{Ann}_a$ is a left ideal.

Now, on $A^n$ we can define the annihilator weight $aw$ that counts the number of components in each orbit. \end{definition}

\textbf{Remark:} It is easily seen that  the EP  for Hamming weight implies  the EP for swc, and the EP for $aw$ as well.

\begin{lemma}\label{lemma1} Let $_RA$ be a pseudo-injective module. Then for any two elements $a$ and $b$ in $A$, $a\approx b$ if and only if $a\sim b$ ($\sim$ corresponds to the action of the whole group $\mathrm{Aut}_R(A)$).\end{lemma}

\begin{proof}If $a\sim b$, this means $a=b\tau$ for some $\tau \in \mathrm{Aut}_R(A)$, and consequently $\mathrm{Ann}_a=\mathrm{Ann}_b$.

Conversely, if $a\approx b$, then we have (as left $R$-modules) $$Ra\cong{} _RR/\mathrm{Ann}_a={} _RR/\mathrm{Ann}_b\cong Rb,$$with $ra\mapsto r+\mathrm{Ann}_a\mapsto rb$. By Proposition 5.1 in \cite{r7}, since $A$ is pseudo-injective, the isomorphism $Ra\rightarrow Rb\subseteq A$ extends to an automorphism of $A$ taking $a$ to $b$.\end{proof}

\begin{corollary}\label{cor1}If $_RA$ is a pseudo-injective module, then the EP with respect to $swc$ built on $\mathrm{Aut}_R(A)$ is equivalent to the EP with respect to $aw$.\end{corollary}

\begin{theorem}\label{Midway}Let $R$ be a principal ideal ring,  $_RA$  a pseudo-injective module, and let $C$ be a submodule of $A^n$ for some $n$. Then a monomorphism  $f:C\rightarrow A^n$ preserves Hamming weight if and only if it preserves swc built on $\mathrm{Aut}_R(A)$.\end{theorem}

\begin{proof} The ``if'' part is direct. For the converse,
we'll use that any left ideal $I$ contains an element $e_I$ that doesn't belong to any other left ideal not containing $I$. Now, if \begin{equation}\label{eq1}(c_1,c_2,\ldots,c_n)f=(b_1,b_2,\ldots,b_n),\end{equation} choose, from $c_1,c_2,\ldots,c_n;b_1,b_2,\ldots,b_n$, a component with a maximal annihilator  $I$. Act on equation (\ref{eq1}) by $e_I$, then the only zero places  are those of the components in equation (\ref{eq1}) with annihilator $I$, and the preservation of Hamming weight gives the preservation of $I$-annihilated components. Omit these components from the list $c_1,c_2,\ldots,c_n;b_1,b_2,\ldots,b_n$ and choose one with the new maximal, and repeat. This gives that $f$ preserves $aw$ and hence, by Lemma \ref{lemma1}, $f$ preserves swc built on $\mathrm{Aut}_R(A)$.
\end{proof}

\begin{corollary}\label{cor2}If $_RA$ is a module alphabet, then $A$ has the extension property with respect to swc if and only if $\mathrm{soc}(A)$ is cyclic.\end{corollary}
\begin{proof} \begin{par}The ``if'' part is answered by Theorem \ref{Noha}. Now, if $\mathrm{soc}(A)$ is not cyclic, then by Proposition \ref{prop}, there is an index $i$ such that $s_i>\mu_i$, where $s_iT_i\subset \mathrm{soc}(A)\subset A$. Recall that $T_i$  is the pullback to $R$ of the matrix module $_{\mathbb{M}_{\mu_i}(\mathbb{F}_{q_i})}\mathbb{M}_{\mu_i\times 1}(\mathbb{F}_{q_i})$, so that $s_iT_i$ is the pullback to $R$ of the $\mathbb{M}_{\mu_i}(\mathbb{F}_{q_i})$-module $B=\mathbb{M}_{\mu_i\times s_i}(\mathbb{F}_{q_i})$. Theorem \ref{wood} implies the existence of linear codes $C_+,C_-\subset B^N$, and an isomorphism $f:C_+\rightarrow C_-$ that preserves Hamming weight, yet $f$ does not extend to a monomial transformation of $B^N$.  But the ring $\mathbb{M}_{\mu_i}(\mathbb{F}_{q_i})$ is a principal ideal ring (in fact, more is true, Theorem ix.3.7, \cite{rw}), besides, $B$ is injective, and then Theorem \ref{Midway} implies that $f$ preserves swc built on $\mathrm{Aut}_{\mathbb{M}_{\mu_i}(\mathbb{F}_{q_i})}(B)$. \end{par}

\begin{par}Now, a little notice finishes the work. 
The isomorphism in equation (\ref{iso1}) and the projection mappings $R\rightarrow R/\mathrm{rad}R\rightarrow\mathbb{M}_{\mu_i}(\mathbb{F}_{q_i})$ allow us to consider the whole situation for $C_\pm$ as $R$-modules. Since $B$ pulls back to $s_iT_i$, we have
$C_\pm\subset(s_iT_i)^N\subset{\mathrm{soc}(A)}^N\subset A^N$, as $R$-modules. Thus $C_\pm$ are linear codes over $A$ that are isomorphic through an isomorphism preserving swc built on $\mathrm{Aut}_R(s_iT_i)$. Also, any automorphism of $A$ restricts to an automorphism of $s_iT_i$, hence the isomorphism preserves swc built on $\mathrm{Aut}_R(A)$. However, this isomorphism does not extend to a monomial transformation of $A^N$, since, as appears in the proof of Theorem \ref{wood} (found in \cite{r6}), $C_+$ has an identically zero component, while $C_-$ does not. \end{par}\end{proof}

 \end{document}